\documentclass[11pt]{amsart}

\usepackage{amsfonts, amssymb, amscd}
\numberwithin{equation}{section}

\usepackage[symbol]{footmisc}

\usepackage{todonotes}

\usepackage{bm}
\usepackage{verbatim}
\usepackage{amssymb}
\usepackage{mathrsfs}
\usepackage{graphicx}
\usepackage{tikz-cd}
\usepackage{subcaption}
\usepackage{listings}
\usepackage{subfiles}
\usepackage[toc,page]{appendix}
\usepackage{mathtools}
\usepackage{comment}
\usepackage{enumerate}
\usepackage{enumitem}
\usepackage[linesnumbered,ruled]{algorithm2e}

\usepackage{graphicx}
\graphicspath{{images/}}

\usepackage{appendix}
\usepackage{hyperref}
\newcommand{\cO}{\mathcal{O}}

\newcommand{\Pp}{\mathbb{P}}
\newcommand{\Qq}{\mathbb{Q}}

\newcommand{\Rr}{\mathbb{R}}

\newcommand{\Zz}{\mathbb{Z}}

\newcommand{\Nn}{\mathbb{N}}

\newcommand{\Exc}{\operatorname{Exc}}

\newcommand{\spec}{\operatorname{Spec}}
\newcommand{\Supp}{\operatorname{Supp}}

\newcommand{\mult}{\operatorname{mult}}

\newcommand{\lf}{\lfloor}
\newcommand{\rf}{\rfloor}

\newcommand{\Oo}{\mathcal{O}}

\newtheorem{thm}{Theorem}[section]

\newtheorem{cor}[thm]{Corollary}
\newtheorem{lem}[thm]{Lemma}

\newtheorem{prop}[thm]{Proposition}

\newtheorem{claim}[thm]{Claim}
\theoremstyle{definition}
\newtheorem{rem}[thm]{Remark}

\newtheorem{set}[thm]{Setting}

\theoremstyle{definition}
\newtheorem{defn}[thm]{Definition}

\begin{document}

\author{Lingyao Xie and Qingyuan Xue}

\address{Department of Mathematics, The University of Utah, Salt Lake City, UT 84112, USA}
\email{lingyao@math.utah.edu}

\address{Department of Mathematics, The University of Utah, Salt Lake City, UT 84112, USA}
\email{xue@math.utah.edu}

\date{\today}

\title[Existence of flips for threefolds in mixed characteristic $(0,5)$]{On the existence of flips for threefolds in mixed characteristic $(0,5)$}
\maketitle
\begin{abstract}
We provide a detailed proof of the validity of the Minimal Model Program for threefolds over excellent Dedekind separated schemes whose residue fields do not have characteristic 2 or 3.
\end{abstract}
\tableofcontents

\section{Introduction}

One of the fundamental goals of algebraic geometry is to classify all algebraic varieties (up to birational equivalence), which, conjecturally, can be achieved by means of the Minimal Model Program (MMP). In characteristic zero, the program holds for varieties with dimension $\le 3$, and a major part of MMP is known for varieties of general type in higher dimensions by \cite{BCHM10}, where they also established the existence of klt flips (see \cite{Bir12,HX13,HL21} for results in a more general setting). In positive characteristic, this theory is now known to hold for threefolds over perfect fields of characteristic $p>3$ (see \cite{HX15,CTX15,Bir16,BW17,GNT19,HW19b}) and in some special cases for fourfolds (\cite{HW20,XX21}). In mixed characteristic, the MMP is known to hold for excellent surfaces (\cite{Tan18}) and semi-stable schemes over excellent Dedekind schemes of relative dimension 2 whose residual characteristics $p\neq 2,3$ (\cite{Kaw94}). Recently substantial progress has been achieved for threefolds. It has been shown that the program is valid for threefolds whose residue fields do not have characteristic 2, 3 or 5 (\cite{BMP+20}). It has also been shown that the MMP holds for strictly semi-stable schemes over excellent Dedekind schemes of relative dimension 2 and for birational morphisms $f$ with $\Exc(f)\subseteq\lf\Delta\rf$ (\cite{TY20}).\par
The goal of this article is to extend the Minimal Model Program for threefolds in mixed characteristic whose residue fields could have characteristic $5$. This is expected to hold as an immediate generalization of \cite{HW19b} (cf. \cite[Remark 9.3]{BMP+20}), but no proof has been written down in detail. Thus we think it may be worthwhile to give the precise statement and its complete proof for future references.

We essentially follow the same strategy of \cite{HW19b}, where they proved the existence of flips for threefolds over an algebraically closed field with characteristic $5$. We generalize their proof to mixed characteristic by using the new techniques developed by \cite{BMP+20} and \cite{TY20}.

\begin{set}\label{setting of V}
In this article, $V$ is an excellent Dedekind scheme whose residue fields do not have characteristic 2 or 3.
\end{set}

\begin{thm}\label{thm: existence of flip for klt pairs}
Let $(X,\Delta)$ be a three-dimensional $\Qq$-factorial klt pair over $V$. If $f:X\to Z$ is a flipping contraction over $V$ such that $\rho(X/Z)=1$, then the flip $f^+:X^+\to Z$ exists.
\end{thm}
Note that this result is known when the residue fields of $V$ do not have characteristic 2,3 or 5 by \cite{BMP+20}. As corollaries of Theorem \ref{thm: existence of flip for klt pairs}, we have the following results on the MMP in mixed characteristic.

\begin{thm}[Minimal Model Program with scaling]\label{thm: MMP with scaling}
Let $(X,\Delta)$ be a three-dimensional $\Qq$-factorial dlt pair over $V$ and let $f:X\to Z$ be a projective contraction over $V$ such that $\dim f(X)>0$. Then we can run a $(K_X+\Delta)$-MMP with scaling of an ample divisor over $Z$. If $K_X+\Delta$ is relatively pseudo-effective, then the MMP terminates with a log minimal model over $Z$. Otherwise, the MMP terminates with a Mori fibre space.
\end{thm}

\begin{thm}[Base point free theorem]\label{thm: Base point free theorem}
Let $(X,\Delta)$ be a three-dimensional $\Qq$-factorial klt pair over $V$ and let $f:X\to Z$ be a projective contraction over $V$ such that $\dim f(X)>0$. Let $D$ be a relatively nef $\Qq$-Cartier $\Qq$-divisor such that $D-(K_X+\Delta)$ is nef and big over $Z$. Then $D$ is semi-ample over $Z$.
\end{thm}

\begin{thm}[Cone theorem]\label{thm: Cone theorem}
Let $(X,\Delta)$ be a three-dimensional $\Qq$-factorial dlt pair over $V$ and let $f:X\to Z$ be a projective surjective contraction over $V$ such that $\dim f(X)>0$. Then there exists a countable number of rational curves $\Gamma_i$ such that 
\begin{enumerate}
    \item $\overline{\mathrm{NE}}(X/Z)=\overline{\mathrm{NE}}(X/Z)_{K_X+\Delta\ge0}+\sum_i\Rr[\Gamma_i]$,
    \item the rays $\Rr[\Gamma_i]$ do not accumulate inside $\overline{\mathrm{NE}}(X/Z)_{K_X+\Delta<0}$, and
    \item for each $\Gamma_i$,
    $$
    -4d_{\Gamma_i}<(K_X+\Delta)\cdot\Gamma_i<0
    $$
    where $d_{\Gamma_i}$ is such that for any Cartier divisor $L$ on $X$, we have $L\cdot\Gamma_i$ divisible by $d_{\Gamma_i}$.
\end{enumerate}
\end{thm}

The above results were proven in \cite[Section 9]{BMP+20} contingent upon the existence of flips with standard coefficients. Hence they follow immediately from Theorem \ref{thm: existence of flip for klt pairs}.

Note that the above results do not require $V$ to be mixed-characteristic. If in addition $V$ is of mixed characteristic, then we actually know the termination of flips.

\begin{thm}[Termination of flips]\label{thm: Termination of flips}
Let $(X,\Delta)$ be a three-dimensional $\Qq$-factorial dlt pair over $V$ and let $f:X\to Z$ be a projective contraction over $V$. Assume that $X_\mathbb{Q} \neq \emptyset$. Then any sequence of $(K_X+\Delta)$-MMP terminates.
\end{thm}

\

\noindent\textbf{Acknowledgement}. The authors would like to thank their advisor Christopher D. Hacon for introducing this question, and giving useful advice and encouragements. The authors also want to thank Jihao Liu and Jingjun Han for useful discussions. The authors were partially supported by NSF research grants no: DMS-1801851, DMS-1952522 and by a grant from the Simons Foundation; Award Number: 256202.

\section{Preliminaries}
A scheme $X$ is called a variety over a field $k$ (resp. a Dedekind scheme $V$) if it is integral, separated, and of finite type over $k$ (resp. $V$). We refer the reader to \cite{KM98} for the standard definitions and results of the Minimal Model Program and to \cite{BMP+20} for those in mixed characteristic. We also refer the readers to \cite{HW19a} for a brief introduction to F-regularity and \cite{BMP+20} for $+$-regularity (which is also called T-regularity in \cite{TY20}).\par
We remark that in this paper, unless otherwise stated, if $(X,B)$ is a pair, then $B$ is a $\Qq$-divisor. We say that $(X,\Delta^c)$ is an $m$-complement of $(X,\Delta)$ if $(X,\Delta^c)$ is log canonical, $m(K_X+\Delta^c)\sim0$, and $\Delta^c\ge\Delta^*$, where $\Delta^*=\frac{1}{m}\lf(m+1)\Delta\rf$. If $\Delta$ has standard coefficients, then $\Delta^*=\frac{1}{m}\lceil m\Delta\rceil$, and so the last condition is equivalent to $\Delta^c\ge\Delta$. We say that a morphism $f:X\to Y$ is a projective contraction if it is a projective morphism of quasi-projective varieties and $f_*\Oo_X=\Oo_Y$.

\begin{set}\label{setting of R}
In this article, $R$ is an excellent local domain with a dualizing complex and positive-characteristic residue field.
\end{set}

\begin{defn}
Let $(X,\Delta)$ be a log canonical pair. We say that $(X,\Delta)$ is qdlt if for every log canonical centre $x\in X$ of codimension $k>0$, there exist distinct irreducible divisors $D_1,...,D_k\subseteq\Delta^{=1}$ such that $x\in W:=D_1\cap...\cap D_k$.
\end{defn}

\begin{rem}[{\cite[Remark 2.4]{HW19b}}]\label{rem: generic point of k-stratum has codim k}
Note that if $(X,\Delta)$ is log canonical and $x$ is a generic point of a stratum $W:=D_1\cap...\cap D_k$ of $\Delta^{=1}$, then $\text{codim}~x=k$.
\end{rem}

\begin{lem}[cf. {\cite[Lemma 2.5]{HW19b}}]
Let $(X,\Delta)$ be a $\Qq$-factorial qdlt pair of dimension $n\le3$ over an excellent Dedekind separated scheme. Then
\item[(1)]
$(D^n,\Delta_{D^n})$ is qdlt, where $g:D^n\to D$ is the normalization of a divisor $D\subseteq \Delta^{=1}$ and $K_{D^n}+\Delta_{D^n}=(K_X+\Delta)|_{D^n}$,
\item[(2)]
the strata of $\Delta^{=1}$ are normal up to a universal homeomorphism, and
\item[(3)]
the log canonical centres of $(X,\Delta)$ coincide with the generic points of strata of $\Delta^{=1}$.
\end{lem}
\begin{proof}
We work in a sufficiently small neighborhood of a point of $X$.\par
First, note that irreducible divisors in $\Delta^{=1}$ are normal up to a universal homeomorphism. Indeed, if $D\subseteq\Delta^{=1}$ is an irreducible divisor, then$(X,\Delta-\lf\Delta\rf+D)$ is plt and hence dlt. Then we can apply \cite[Lemma 2.28]{BMP+20}.\par
Let $x\in D^n$ be a log canonical centre of $(D^n,\Delta_{D^n})$. Then $g(x)$ is a log canonical centre of $(X,\Delta)$. Indeed, otherwise there exist a non-zero divisor $H$ passing through $g(x)$ and $\epsilon>0$ such that $(X,\Delta+\epsilon H)$ is lc at $g(x)$. Thus, by adjuntion, $(D^n,\Delta_{D^n}+\epsilon H|_{D^n})$ is lc at $x$, which is impossible.\par
Let $k$ be the codimension of $g(x)$ in $X$. By definition of qdlt pairs, there exist divisors $D_1,...,D_k\subseteq\Delta^{=1}$ with $D_1=D$, such that
$$
g(x)\in D_1\cap...\cap D_k.
$$
Then $x\in D_2|_{D^n}\cap...\cap D_k|_{D^n}$, where $D_i|_{D^n} \subseteq \Delta^{=1}_{D^n}$ and $D_i|_{D^n}$ and $D_j|_{D^n}$ have no common components for $i, j \ge 2$. Since $x$ is of codimension $k-1$ in $D^n$, this shows that $(D^n,\Delta_{D^n})$ is qdlt at $x$. Hence (1) holds.\par
As for (2) and (3), they can be proven by induction on the dimension $n$ and the fact that $D$ is normal up to a universal homeomorphism. 
\end{proof}

\begin{lem}[Inversion of adjunction]\label{lem: inversion of adjunction for qdlt pair}
Consider a three-dimensional $\Qq$-factorial log pair $(X,S+E+B)$ over an excellent Dedekind separated scheme, where $S,E$ are irreducible divisors and $\lf B\rf=0$. Write $K_{S^n}+C_{S^n}+B_{S^n}=(K_X+S+E+B)|_{S^n}$, where $S^n$ is the normalisation of $S$, $C_{S^n}=(E\cap S)|_{S^n}$ is an irreducible divisor, and $\lf B_{S^n}\rf=0$. Assume that $(S^n,C_{S^n}+B_{S^n})$ is plt. Then $(X,S+E+B)$ is qdlt in a neighborhood of $S$.
\end{lem}
\begin{proof}
Assume by contradiction that $(X,S+E+B)$ admits a log canonical centre $Z$ of codimension at least two, which is different from $C=E\cap S$ and intersects $S$. Let $H$ be a general Cartier divisor containing $Z$. Then for any $0<\delta\ll 1$ we can find $0<\epsilon\ll1$ such that $(X,S+(1-\epsilon)E+B+\delta H)$ is not lc at $Z$. On the other hand, $(S^n,(1-\epsilon')C_{S^n}+B_{S^n}+\delta H|_{S^n})$ is klt for any $0<\epsilon'\ll1$. This contradicts \cite[Corollary 4.10]{TY20}.
\end{proof}

\begin{lem}[{\cite[Lemma 2.7]{HW19b}}]\label{lem: S_1 and S_2 has no intersection after flop if not qdlt}
Let $(X,S_1+S_2+B)$ be a three-dimensional $\Qq$-factorial qdlt pair where $S_1,S_2$ are irreducible divisors and $\lf B\rf=0$. Let $$
f:(X,S_1+S_2+B)\dasharrow(X',S'_1+S'_2+B')
$$
be a $(K_X+S_1+S_2+B)$-flop of a curve $\Sigma$ for a relative-Picard-rank-one flopping contraction $g:X\to Z$. Suppose that $S_1\cdot\Sigma<0$. Then either $(X',S'_1+S'_2+B')$ is qdlt or $S'_1\cap S'_2=\emptyset$ in a neighbourhood of $\Exc(g')$, where $g':X'\to Z$ is the flopped contraction. 
\end{lem}

\begin{lem}[{\cite[Lemma 7.13]{BMP+20}}]\label{lem: extract kollar component}
Let $(X,B)$ be a two-dimensional klt pair admitting a projective birational map $f:X\to Z=\spec R$ such that $-(K_X+B)$ is relatively nef, assuming that $R$ is as in Setting \ref{setting of R} and additionally has infinite residue field. Then there exist an $f$-exceptional irreducible curve $C$ on a blow-up of $X$ and projective birational maps $g:Y\to X$ and $h:Y\to W$ over $Z$ such that:
\begin{enumerate}
    \item $g$ extracts $C$ or is the identity if $C\subseteq X$,
    \item $(Y,C+B_Y)$ is plt,
    \item $(W,C_W+B_W)$ is plt and $-(K_W+C_W+B_W)$ is ample over $Z$,
    \item $h^*(K_W+C_W+B_W)-(K_Y+C+B_Y)\ge0$,
\end{enumerate}
where $K_Y+bC+B_Y=g^*(K_X+B)$ for $C\nsubseteq \Supp B_Y$, $C_W:=h_*C\neq0$, and $B_W := h_* B_Y$.
\end{lem}

\begin{lem}[{\cite[Theorem 7.14]{BMP+20}}]
Let $(X,B)$ be a two-dimensional klt pair admitting a projective birational map $f:X\to Z=\spec R$ such that $-(K_X+B)$ is relatively ample. Suppose that $R$ is as in Setting \ref{setting of R} and has residual characteristic $p>5$, and that $B$ has standard coefficients. Then $(X,B+\epsilon D)$ is globally $+$-regular over $Z$ for every effective divisor $D$ and $0\le\epsilon\ll1$.
\end{lem}

\begin{rem}\label{rem: (2,3,5) case in char 5}
If $p=5$, then the above proposition remains true unless $B_C=\frac{1}{2}P_1+\frac{2}{3}P_2+\frac{4}{5}P_3$ for three distinct points $P_1,P_2$ and $P_3$.
\end{rem}

In what follows we need an analogue of \cite[Lemma 2.11]{HW19b} in mixed characteristic. The proof is similar except that we need extra consideration in the last of the proof.

\begin{lem}[cf. {\cite[Lemma 2.11]{HW19b}}]\label{lem: 6-complement with an integral boundary}
With notation as in Lemma \ref{lem: extract kollar component}, suppose that $p>3$ and  $(X,B)$ admits a 6-complement $(X,E+B^c)$, where $E$ is a non-exceptional irreducible curve intersecting the exceptional locus over $Z$. Then for any effective divisor $D$, $(X,B+\epsilon D)$ is globally $+$-regular over $Z$ for any $0\le\epsilon\ll1$.
\end{lem}
\begin{proof}
As in the proof of \cite[Theorem 7.14]{BMP+20}, it is enough to show that $(C_{\bar{k}},B_{C_{\bar{k}}})$ is globally F-regular, where $C$ is the exceptional curve in Lemma \ref{lem: extract kollar component}, $K_C+B_C=(K_W+C_W+B_W)|_C$ and $k=H^0(C,\Oo_C)$.\par
By pulling back the complement to $Y$ and pushing down on $W$, we obtain a sub-lc pair $(W,aC_W+E_W+B_W^c)$ for a (possibly negative) number $a\in\Qq$ such that $6(K_W+aC_W+E_W+B^c_W)\sim_{Z}0$, a non-exceptional irreducible curve $E_W$ intersecting the exceptional locus over $Z$, and an effective $\Qq$-divisor $B_W^c$ such that $E_W+B^c_W\ge B_W$. Let $T_W$ be an effective exceptional anti-ample $\Qq$-divisor on $W$ and let $\lambda\ge0$ be such that the coefficient of $C_W$ in $aC_W+\lambda T_W$ is one. By the Koll\'ar-Shokurov connectedness theorem (see e.g. \cite[Theorem 5.2]{Tan18}), the pair $(W,aC_W+\lambda T_W+E_W+B^c_W)$ is not plt along $C_W$. In particular, $B_C^c$ contains a point with coefficient at least one, where
$$
(K_W+aC_W+\lambda T_W+E_W+B_W^c)|_C=K_C+B^c_C.
$$
Since $T_W$ is anti-ample over $Z$, we have that $K_C+B^c_C$ is anti-nef. In particular, there exists a $\Qq$-divisor $B_C\le B'_C\le B^c_C$ such that $(C,B'_C)$ is plt (but not klt) and $-(K_C+B'_C)$ is nef.\par
Now we claim that $(C_{\bar{k}},B'_{C_{\bar{k}}})$ is plt (but not klt), where $B'_{C_{\bar{k}}}:=(B'_C)_{\bar{k}}$. Indeed, since $C_{\bar{k}}\cong\Pp^1_{\bar{k}}$ and $K_{C_{\bar{k}}}+B'_{C_{\bar{k}}}$ is anti-nef, we have $\deg_{\bar{k}}B'_{C_{\bar{k}}}\le 2$. Noting that any coefficient of $B'_{C_{\bar{k}}}$ is either equal to the corresponding coefficient of $B'_C$ or at least $p$ times that coefficient with $p>3$, we can then easily deduce that $\lf B'_{C_{\bar{k}}}\rf=(\lf B'_C\rf)_{\bar{k}}\neq 0$ has coefficient one for each irreducible component and that $\lf(\{B'_C\})_{\bar{k}}\rf=0$, which implies our claim.\par
If $-(K_{C_{\bar{k}}}+B'_{C_{\bar{k}}})$ is ample, then $(C_{\bar{k}},B'_{C_{\bar{k}}})$ is purly F-regular by \cite[Lemma 2.9]{CTW17} (applied to perturbations of $(C_{\bar{k}},B'_{C_{\bar{k}}})$), and so $(C_{\bar{k}},B_{C_{\bar{k}}})$ is globally F-regular. If $-(K_{C_{\bar{k}}}+B'_{C_{\bar{k}}})$ is trivial, then $a=1,\lambda=0,~ 6(K_{C_{\bar{k}}}+B^c_{C_{\bar{k}}})\sim0$, and $(C_{\bar{k}},B^c_{C_{\bar{k}}})$ is plt (but not klt). Since gcd$(p,6)=1$, \cite[Lemma 2.9]{CTW17} implies that $(C_{\bar{k}},B^c_{C_{\bar{k}}})$ is globally F-split, and so $(C_{\bar{k}},B_{C_{\bar{k}}})$ is globally F-regular by \cite[Corollary 3.10]{SS10}.
\end{proof}

\begin{defn}
Let $(X,\Delta)$ be a three-dimensional dlt pair. We define its dual complex $D(\Delta^{=1})$ to be a simplex with nodes corresponding to irreducible divisors of $\Delta^{=1}$ and $k$-simplices between $k+1$ nodes corresponding to $k+1$ divisors containing a common codimension $k+1$ locus. We say that an irreducible divisor $D$ in $\Delta^{=1}$ is an articulation point of $D(\Delta^{=1})$ if $\Delta^{=1}-D$ is disconnected.
\end{defn}

\begin{lem}\label{lem: articulation point is stable between birational morphism}
Let $(X,\Delta)$ be a $\Qq$-factorial dlt threefold over an excellent Dedekind sparated scheme and let $\pi:Y\to X$ be a projective birational morphism such that $(Y,\pi_*^{-1}\Delta+E)$ is dlt, where $E$ is the exceptional locus of $\pi$. Write $K_Y+\Delta_Y=\pi^*(K_X+\Delta)$. Let $S$ be an irreducible divisor in $\Delta^{=1}$, and let $S_Y$ be its strict transform. If $S_Y$ is an articulation point of $D(\Delta^{=1}_Y)$, then $S$ is an articulation point of $D(\Delta^{=1})$.
\end{lem}
\begin{proof}
This follows exactly by the same proof of \cite[Lemma 2.12]{HW19b}, except that in the beginning we use \cite[Theorem 1.2]{TY20} to run a $(K_Y+\pi_*^{-1}\Delta+E)$-MMP over $X$.
\end{proof}

\begin{lem}[cf. {\cite{Wit21}}]\label{lem: semiample can be checked up to uni. homeo. in mixed char}
Let $f:Y\to X$ be a finite universal homeomorphism of schemes which are proper over a Noetherian base scheme $S$. Let $L$ be a nef line bundle on $X$ such that $f^*L$ and $L|_{X_{\Qq}}$ is semiample, where $X_\Qq$ is the generic fiber of $X\to\spec\Zz$. Then $L$ is semiample.
\end{lem}
\begin{proof}
By \cite[Theorem 1.2]{Wit21}, it is enough to verify that $L|_{X_s}$ is semiample for any $s\in S$ whose residue field has positive characteristic.\par
Note that $f^*L|_{Y_s}$ is semiample and  the base change $f_s:Y_s\to X_s$ is a finite universal homeomorphism proper over a field with positive characteristic, we can deduce that $L|_{X_s}$ is semiample by \cite[Lemma 2.11(3)]{CT20}.  
\end{proof}

\section{Complements on surfaces}
\begin{prop}\label{prop: unique non-klt place for 6-complement}
Let $(X,B)$ be a two-dimensional klt pair admitting a projective birational map $f:X\to Z=\spec R$ such that $-(K_X+B)$ is relatively nef but not numerically trivial, where $R$ is as in Setting \ref{setting of R} and additionally has infinite residue field with characteristic $p>3$. Assume that there exists an effective divisor $D$ such that $(X,B+\epsilon D)$ is not globally $+$-regular over $Z$ for any $\epsilon>0$. \par
Then every 6-complement of $(X,B)$ is non-klt and has a unique non-klt valuation which is exceptional over $Z$.

\end{prop}
\begin{proof}
By Lemma \ref{lem: extract kollar component}, there exist an irreducible, exceptional over $Z$, curve $C$ on a blow-up of $X$  and projective birational maps $g:Y\to X$ and $h:Y\to W$ over $Z$ such that
\begin{enumerate}
    \item $g$ extracts $C$ or is the identity if $C\subseteq X$,
    \item $(Y,C+B_Y)$ is plt,
    \item $(W,C_W+B_W)$ is plt and $-(K_W+C_W+B_W)$ is ample over $Z$,
\end{enumerate}
where $C_W:=h_*C\neq 0$, $B_W:=h_*B_Y$, and $K_Y+bC+B_Y=g^*(K_X+B)$ for $C\nsubseteq\Supp B_Y$.

By Remark \ref{rem: (2,3,5) case in char 5}, $(K_W+C_W+B_W)|_{C_W}=K_{C_W}+\frac{1}{2}P_1+\frac{2}{3}P_2+\frac{4}{5}P_3$ for some three distinct points $P_1,P_2$ and $P_3$.\par
\vspace{1em}
Now, let $(X,B^c)$ be any 6-complement of $(X,B)$. By the negativity lemma $\Supp(B^c-B)$ contains a non-exceptional curve. Let $K_Y+aC+B^c_Y=g^*(K_X+B^c)$, where $C\nsubseteq\Supp B^c_Y$, and let $B_W^c:=h_*B_Y^c$. Since $6(K_X+B^c)\sim_Z0$ is lc, we get that
$$
(W,aC_W+B^c_W) 
$$
is a sub-lc and $6(K_W+aC_W+B^c_W)\sim_Z0$. In particular $6B^c_W$ is an integral divisor. Moreover, $B_W^c\ge B_W$ as $B^c\ge B$.\par
\vspace{1em}
To prove the proposition it is now enough to show that $a=1$. Indeed, in this case $-(K_W+C_W+B^c_W)\sim_{\Qq,Z}0$ and by the Koll\'ar-Shokurov connectedness lemma, the non-klt locus of $(W,C_W+B^c_W)$ is connected. The only 6-complement of 
$$
(C_W,\frac{1}{2}P_1+\frac{2}{3}P_2+\frac{4}{5}P_3)
$$
is $(C_W,\frac{1}{2}P_1+\frac{2}{3}P_2+\frac{5}{6}P_3)$, so $(W,C_W+B_W^c)$ is plt along $C_W$ by adjunction, and the connectedness of non-klt locus implies that $(W,C_W+B^c_W)$ is in fact plt everywhere. In particular, $(X,B^c)$ admits a unique exceptional non-klt valuation over $Z$.\par
In order to prove the propositon, we assume that $a<1$ and derive a contradiction. We will not need to refer to $(X,B)$ or $(Y,aC+B_Y)$ any more, so, for ease of notation, we replace $C_W,B_W$ and $B_W^c$ by $C,B$ and $B^c$ respectively.\par
If $(B^c-B)\cdot C\neq0$, then Lemma \ref{lem: gap is bigger than 1/30} applied to $(W,C+B^c)$ implies that $(K_W+C+B^c)\cdot C=0$. This is impossible, because
$$
(K_W+C+B^c)\cdot C<(K_W+aC+B^c)\cdot C=0
$$
Hence, we can assume that $(B^c-B)\cdot C=0$. Since $\Supp(B^c-B)$ contains a non-exceptional curve, the exceptional locus over $Z$ cannot be irreducible, and so there exists an irreducible exceptional curve $E\neq C$ such that $E\cap C\neq\emptyset$. Since $K_W+C+B$ is anti-ample over $Z$ and $E$ is an extremal ray of $\overline{\text{NE}}(X/Z)$, we may contract $E$ over $Z$ by \cite[Theorem 4.4]{Tan18}. Let $f:W\to W_1$ be the contraction of $E$, and let $C_1,B^c_1$ be the strict transforms of $C$ and $B^c$. We have that 
$$
(K_W+C+B^c)\cdot E>(K_W+aC+B^c)\cdot E=0,
$$
and hence for some $t>0$ and with the natural identification $C\cong C_1$:
\begin{align*}
    (K_{W_1}+C_1+B^c_1)|_{C_1}&=f^*(K_{W_1}+C_1+B^c_1)|_C \\
                              &=(K_W+C+B^c+tE)|_C \\
                              &\le K_C+\frac{1}{2}P_1+\frac{2}{3}P_2+\frac{4}{5}P_3+tE|_C
\end{align*}
As before, $(K_{W_1}+C_1+B^c_1)\cdot C_1<(K_{W_1}+aC_1+B^c_1)\cdot C_1=0$. By applying Lemma \ref{lem: gap is bigger than 1/30} to $(W_1,C_1+B^c_1)$, we get a contradiction again.
\end{proof}

In the following result, it is key that $\Delta$ is non-zero.

\begin{lem}\label{lem: gap is bigger than 1/30}
Let $(S,C+B)$ be a two-dimensional log pair where $S$ is a normal excellent surface. Let $f:S\to T$ be a projective birational morphism such that the irreducible normal divisor $C$ is exceptional and $(K_S+C+B)\cdot C\le0$. Assume that $6B$ is an integral divisor and 
$$
B_C=\frac{1}{2}P_1+\frac{2}{3}P_2+\frac{4}{5}P_3+\Delta
$$
for distinct points $P_1,P_2,P_3\in C$ and a non-zero effective $\Qq$-divisor $\Delta$, where $(K_S+C+B)|_C=K_C+B_C$. Then $(K_S+C+B)\cdot C=0$.
\end{lem}
\begin{proof}
This follows by exactly the same proof of \cite[Lemma 3.2]{HW19b}. Notice that it only uses the classification of plt singularities for excellent surfaces (see e.g. \cite[Corollary 3.45, 3.33, 3.35 and 3.36]{Kol13}).
\end{proof}

\section{Lifting complements}
The main theorem we want to show in this section is:
\begin{prop}\label{prop: m-complement for (X,S+B)}
Let $(X,S+B)$ be a 3-dimensional $\Qq$-factorial plt pair with standard coefficients, and let $f:X\to Z=\spec R$ be a flipping contraction such that $-(K_X+S+B)$ and $-S$ are $f$-ample. Here the ring $R$ has residual characteristic $p>2$. \par
Then there exists an m-complement $(X,S+B^c)$ of $(X,S+B)$ in a neighborhood of $\Exc f$ for some $m\in\{1,2,3,4,6\}$.
\end{prop}

Since standard coefficients are not stable under log pull-backs, we need to work in a more general setting.

\begin{set}\label{set:stable sections}
Fix a natural number $m\in\Nn$. Let $(X,S+B)$ be a sub-log pair projective over $Z=\spec R$ where $R$ is as in the Setting \ref{setting of R}, such that $S$ is a (possibly empty) reduced Weil divisor, $\lf B\rf\le0$, and $A:=-(K_X+S+B)$ is semi-ample and big.

We are ready to define:
\begin{align*}
    \Phi&:=S+\{(m+1)B\},\\
    D&:=\lceil mB\rceil-\lf(m+1)B\rf, and \\
    L&:=\lf mA\rf+D.
\end{align*}

\end{set}
Notice that $L-(K_X+\Phi)=(m+1)A$ is semi-ample and big and $D=0$ if $B+S$ has standard coefficients.

\begin{lem}\label{lem: B^0 in higher model}
With notation as in Setting \ref{set:stable sections}, suppose that $(X,S+B)$ is plt and $D=0$. Let $\pi:Y\to X$ be a projective birational map and set $K_Y+S_Y+B_Y=\pi^*(K_X+S+B)$ with $S_Y=\pi^{-1}_*S$. Then,
$$
\mathbf{B}^0_{S_Y}(Y,\Phi_Y;L_Y)=\mathbf{B}^0_{S}(X,\Phi;L),
$$
where $L_Y$ and $\Phi_Y$ is defined for $(Y,S_Y+B_Y)$ as in Setting \ref{set:stable sections}.
\end{lem}
\begin{proof}
For any alteration $g:W\to Y$ such that $W$ is normal, let $S_W$ be a strict transform of $S_Y$. We have the following commutative diagram:
\begin{center}
\begin{tikzcd}
H^0(W,K_W+S_W+\lceil g^*( L_Y-K_Y-\Phi_Y)\rceil) \arrow[r] \arrow[d, "\phi"]
& H^0(Y,L_Y) \arrow[d, "\psi"] \\
H^0(W,K_W+S_W+\lceil h^*(L-K_X-\Phi)\rceil) \arrow[r]
& H^0(X,L)
\end{tikzcd}
\end{center}
where $h=\pi\circ g$ and the horizontal maps are trace maps. Since $g^*( L_Y-K_Y-\Phi_Y)=h^*(L-K_X-\Phi)=h^*((m+1)A)$, we see that $\phi$ is actually the identity. Since $\pi_*L_Y=L$ and $L_Y\ge \pi^*L+D_Y$, $\psi$ is an isomorphism. 
\end{proof}

The following lemma allows for lifting sections.

\begin{lem}\label{lem: restriction is surjective}
With notation as in Setting \ref{set:stable sections}, suppose that $(X,S+B)$ is plt with standard coefficients, $S$ is an irreducible divisor, and $A:=-(K_X+S+B)$ is ample. Write $A_{S^n}:=-(K_{S^n}+B_{S^n})=-(K_X+S+B)|_{S^n}$ for the normalisation $S^n$ of $S$. Then by restricting sections we get a surjection
$$
\mathbf{B}^0_{S}(X,\Phi;\lf mA\rf)\to\mathbf{B}^0(S^n,\Phi_{S^n};\lf mA_{S^n}\rf).
$$
\end{lem}

\begin{proof}
Let $\pi:Y\to X$ be a log resolution of $(S+B)$. We can write 
\begin{align*}
    K_Y+S_Y+B_Y&=\pi^*(K_X+S+B),\mathrm{\ and}\\
    K_{S_Y}+B_{S_Y}&=(K_Y+S_Y+B_Y)|_{S_Y}
\end{align*}
for $S_Y=\pi_*^{-1}S$. Define $L_Y,L_{S_Y},\Phi_Y,\Phi_{S_Y}$ as in Setting \ref{set:stable sections}. Then we have $(K_Y+\Phi_Y)|_{S_Y}=K_{S_Y}+\Phi_{S_Y}$ and $L_Y|_{S_Y}=L_{S_Y}$. 

Since $L_Y-(K_Y+\Phi_Y)=-(m+1)\pi^*(K_X+S+B)$ is big and semi-ample, restricting sections induces a surjective map 
$$
\mathbf{B}^0_{S_Y}(Y,\Phi_Y;L_Y)\to\mathbf{B}^0(S_Y,\Phi_{S_Y};L_{S_Y})
$$
by \cite[Theorem 7.2]{BMP+20}. Thus the claim follows from Lemma \ref{lem: B^0 in higher model} applied to both sides.
\end{proof}

Finally, we show that $B^0_S$ gets smaller when the boundary getts bigger.

\begin{lem}\label{lem: B^0 get smaller when boundary gets bigger}
Let $(X,S+B)$ and $(X,S'+B')$ be two sub-log pairs satisfying the assumptions of Setting \ref{set:stable sections}. Suppose that $S'+B'\ge S+B$ and define $\Phi,L$ and $\Phi',L'$ for $(X,S+B)$ and $(X,S'+B')$, respectively, as in Setting \ref{set:stable sections}.

Then $L-L'\ge0$ and the inclusion $H^0(X,L')\subseteq H^0(X,L)$ induces an inclusion
$$
\mathbf{B}^0_{S'}(X,\Phi';L')\subseteq\mathbf{B}^0_S(X,\Phi;L),
$$
\end{lem}
\begin{proof}
First we have
\begin{align*}
    L-L'&=\Phi-\Phi'+(m+1)(S'+B'-S-B)\\
        &=S-S'+\lf(m+1)(S'+B')\rf-\lf(m+1)(S+B)\rf,
\end{align*}
and so $L-L'\ge0$. \par

Note that $S'+B'\ge S+B$ implies $S'\ge S$ and $S'-S\subseteq\Supp(S'+B'-S-B)$. Thus for a sufficiently large finite cover $f:W\to X$, denoting by $S_W$ and $S'_W$ the strict transforms of $S$ and $S'$ such that $S_W\le S'_W$, we have $S_W+f^*(-(m+1)(K_X+S+B))\ge S'_W+f^*(-(m+1)(K_X+S'+B'))$, which is equivalent to $S'_W-S_W\ge f^*((m+1)(S'+B'-S-B))$. Then the statement follows by the definition of $\mathbf{B}^0_S$ (see \cite[Lemma 4.24]{BMP+20}) since $f^*(L-K_X-\Phi)=f^*(-(m+1)(K_X+S+B))$.
\end{proof}

We need the following lemma for the proof of Proposition \ref{prop: m-complement for (X,S+B)}.

\begin{lem}\label{lem: m-complement for surface}
Let $(X,B)$ be a two-dimensional klt pair with standard coefficients admitting a projective birational map $f:X\to Z=\spec R$ such that $-(K_X+B)$ is relatively ample, assuming $R$ is as in Setting \ref{setting of R} and additionally has infinite residue field. Then there exists $m\in\{1,2,3,4,6\}$ and
$$
s\in\mathbf{B}^0(X,\Phi;L)\subseteq H^0(X, L)
$$ 
such that $(X,\frac{1}{m}\lceil mB\rceil+\frac{1}{m}\Gamma)$ is an m-complement of $(X,B)$ in a neighborhood of $\Exc f$ where $\Gamma$ is the divisor corresponding to $s$, and $L$ and $\Phi$ are defined as in Setting \ref{set:stable sections}.
\end{lem}

\begin{proof}
By Lemma \ref{lem: extract kollar component}, there exist an irreducible, exceptional over $Z$, curve $C$ on a blow-up of $X$ and projective birational map $g: Y\to X$ and $h: Y\to W$ over $Z$ such that
\begin{enumerate}
    \item $g$ extracts $C$ or is the identity if $C\subseteq X$,
    \item $(Y,C+B_Y)$ is plt,
    \item $(W,C_W+B_W)$ is plt and $-(K_W+C_W+B_W)$ is ample over $Z$,
    \item $B_Y^+-B_Y\ge0$,
\end{enumerate}
where $K_Y+bC+B_Y=g^*(K_X+B)$ for $C\nsubseteq\Supp B_Y$, $C_W:=h_*C\neq0$, $B_W:=h_*B_Y$, and $K_Y+C+B_Y^+=h^*(K_W+C_W+B_W)$.

We have 
\begin{align*}
    \mathbf{B}^0(X,\Phi;L)&=\mathbf{B}^0(Y,\Phi_Y;L_Y)\\
    &\supseteq \mathbf{B}^0_C(Y,\Phi^+_Y;L^+_Y)\\
    &=\mathbf{B}^0_{C_W}(W,\Phi_W;L_W),
\end{align*}
where $\Phi_Y,\Phi^+_Y,\Phi_W$ and $L_Y,L_Y^+,L_W$ are defined as in Setting \ref{set:stable sections}. Indeed, the first and third equality hold by Lemma \ref{lem: B^0 in higher model} since $B$ and $C_W+B_W$ have standard coefficients, and the middle inclusion holds by Lemma \ref{lem: B^0 get smaller when boundary gets bigger} since $C+B_Y^+\ge bC+B_Y$.

Note that $L=-m(K_X+\frac{1}{m}\lceil mB\rceil)$ and $L_W=-m(K_W+C_W+\frac{1}{m}\lceil mB_W\rceil)$. Thus by Lemma \ref{lem: restriction is surjective}, restricting sections gives a surjective map
$$
\mathbf{B}^0_{C_W}(W,\Phi_W;L_W)\to\mathbf{B}^0(C,\Phi_C;L_C),
$$
where $C$ is identified with $C_W$, $K_C+B_C=(K_W+C_W+B_W)|_C$, and $\Phi_C, L_C$ are defined as in Setting \ref{set:stable sections}.

Let $m\in\{1,2,3,4,6\}$ be the minimal number such that $(C,B_C)$ admits an $m$-complement. 

Since $-(K_C+B_C)$ is ample and $B_C$ has standard coefficients, we must have that $C_{\bar{k}}=\Pp^1_{\bar{k}}$, and the coefficients of $(B_C)_{\bar{k}}=B_{C_{\bar{k}}}$ must exactly be the same as the coefficients of $B_C$. This is because any coefficient of $B_{C_{\bar{k}}}$ is either equal to the corresponding coefficient of $B_C$ or at least $p$ times such a coefficient (hence it is at least $\frac{p}{2}$), and the existence of the latter type of coefficients would contradict the ampleness of $-(K_{C_{\bar{k}}}+B_{C_{\bar{k}}})$. Therefore we have 
$$(\Phi_C)_{\bar{k}}=(\{(m+1)B_C\})_{\bar{k}}=\{(m+1)B_{C_{\bar{k}}}\}.$$
By \cite[Lemma 4.9]{HW19b}, $(C_{\bar{k}},(\Phi_C)_{\bar{k}})$ is globally F-regular, and hence $(C,\Phi_C)$ is globally $+$-regular by \cite[Corollary 6.17]{BMP+20}. Therefore
$$
\mathbf{B}^0(C,\Phi_C;L_C)=H^0(C,L_C).
$$
In particular, there exists an lc $m$-complement $(C,B^c_C)$ of $(C,B_C)$ for some $m\in\{1,2,3,4,6\}$ which can be lifted to $W$. More precisely, there exists a non-zero section
$$
s\in\mathbf{B}^0_{C_W}(W,\Phi_W;L_W)
$$
with associated divisor $\Gamma$ such that $m(K_W+C_W+B_W^c)\sim0$ and 
$$
(K_W+C_W+B_W^c)|_C=K_C+B_C^c,
$$
where $B_W^c:=\frac{1}{m}\lceil mB_W\rceil+\frac{1}{m}\Gamma$. By inversion of adjunction, $(W,C_W+B_W^c)$ is log canonical along $C_W$. Note that 
$$
K_W+C_W+\epsilon B_W+(1-\epsilon)B_W^c
$$
is thus plt along $C_W$ and $\Qq$-equivalent over $Z$ to $\epsilon(K_W+C_W+B_W)$, and hence by Koll\'ar-Shokurov connectedness principle (cf. \cite[Theorem 5.2]{Tan18}), it is plt for any $0<\epsilon<1$. Hence $(W,C_W+B_W^c)$ is lc, and thus an $m$-complement of $(W,C_W+B_W)$.

Let $K_Y+C+B_Y^c=h^*(K_W+C_W+B_W^c)$ and $B^c:=g_*(C+B_Y^c)$. Then $(X,B^c)$ is an $m$-complement of $(X,B)$ which by the above inclusions of $\mathbf{B}^0$ corresponds to a section in $\mathbf{B}^0(X,\Phi;L)$.
\end{proof}

\begin{proof}[Proof of Proposition \ref{prop: m-complement for (X,S+B)}]
Let $S^n$ be the normalisation of $S$. By Lemma \ref{lem: restriction is surjective}, restricting sections gives a surjective map
$$
\mathbf{B}^0_{S}(X,\Phi;\lf mA\rf)\to\mathbf{B}^0(S^n,\Phi_{S^n};\lf mA_{S^n}\rf),
$$
notice that  $\lf mA\rf=-m(K_X+S+\frac{1}{m}\lceil mB\rceil)$ and $\lf mA_{S^n}\rf=-m(K_{S^n}+\frac{1}{m}\lceil mB_{S^n}\rceil)$.\par
By Lemma \ref{lem: m-complement for surface}, there exists $\Gamma_{S^n}\in| -m(K_{S^n}+\frac{1}{m}\lceil mB_{S^n}\rceil)|$ such that $(S^n,B^c_{S^n})$ is an m-complement of $(S^n,B_{S^n})$ for $B^c_{S^n}=\frac{1}{m}\lceil mB_{S^n}\rceil+\frac{1}{m}\Gamma_{S^n}$, and which moreover lifts to 
$$
\Gamma\in|-m(K_X+S+\frac{1}{m}\lceil mB\rceil)|.
$$
Set $B^c=\frac{1}{m}\lceil mB\rceil+\frac{1}{m}\Gamma$. Then $m(K_X+S+B^c)\sim0$ and $(K_X+S+B^c)|_{S^n}=K_{S^n}+B^c_{S^n}$. By inversion of adjunction (\cite[Corollary 4.10]{TY20}) applied to $(X,S+(1-\epsilon)B^c)$ for $0<\epsilon\ll1$, we get that $(X,S+B^c)$ is lc in a neighborhood of $\Exc f$, and hence it is an $m$-complement of $(X,S+B)$.
\end{proof}

\begin{rem}\label{rem: m=6 when char p=5}
With notation as in Proposition \ref{prop: m-complement for (X,S+B)}, if the residue field has characteristic $p=5$ and there exists an effective divisor $D$ such that $(S^n,B_{S^n}+\epsilon D)$ is not globally $+$-regular over $Z$ for any $\epsilon>0$, where $S^n$ is the normalisation of $S$ and $K_{S^n}+B_{S^n}=(K_X+S+B)|_{S^n}$, then $m=6$. 
\end{rem}
\begin{proof}
Under these assumptions, we see that in the proof of Lemma \ref{lem: m-complement for surface} $B_C=\frac{1}{2}P_1+\frac{2}{3}P_2+\frac{4}{5}P_3$ for three distinct points $P_1,P_2$ and $P_3$ by Remark \ref{rem: (2,3,5) case in char 5}. The smallest $m$ such that this $(C,B_C)$ admits an $m$-complement is $m=6$.
\end{proof}

\section{Flips admitting a qdlt complement}
The goal of this section is to show that the existence of flips for flipping contractions admitting a qdlt $k$-complement, where $k\in\{1,2,3,4,6\}$.

\begin{prop}\label{Prop: Existence of flip for qdlt pair}
Let $(X,\Delta)$ be a $\Qq$-factorial qdlt 3-dimensional pair with standard coefficients over $V$. Let $f:X\to Z$ be a $(K_X+\Delta)$-flipping contraction over $V$ such that $\rho(X/Z)=1$ and let $\Sigma$ be a flipping curve. Assume that there exists a qdlt 6-complement $(X,\Delta^c)$ of $(X,\Delta)$ such that $\Sigma\cdot S<0$ for some irreducible divisor $S\subseteq\lf\Delta^c\rf$. Then the flip $f^+:X^+\to Z$ exists.
\end{prop}
\begin{proof}

Write $\Delta=aS+D+B$, where $1\ge a\ge0$, the divisor $D$ is integral, $S\nsubseteq\Supp(D+B)$, and $\lf B\rf=0$. By replacing $\Delta$ by $S+(1-\frac{1}{k}D+B)$ for $k\gg0$, we can assume that $(X,\Delta)$ is plt. Then we can split the proof into three cases:
\begin{enumerate}
    \item $(X,\Delta^c)$ is plt along the flipping locus, or
    \item $\Sigma\cdot E<0$ for a divisor $E\subseteq\lf\Delta^c\rf$ different from $S$, or
    \item $\Sigma\cdot E\ge0$ for a divisor $E\subseteq\lf\Delta^c\rf$ intersecting the flipping locus.
\end{enumerate}
Case (1) and Case (3) follow from Proposition \ref{prop: case 1} and Proposition \ref{prop: case 3} respectively, applied to $(X,\Delta)$. Case (2) follow from Propostion \ref{prop: case 2} applied to $(X,\Delta+bE)$ where $b\ge0$ is such that $\mult_E(\Delta+bE)=1$.
\end{proof}

\begin{prop}\label{prop: case 1}
Let $(X,S+B)$ be a 3-dimensional $\Qq$-factorial plt pair over $V$ with $S$ is irreducible and $B$ having standard coefficients. Let $f:X\to Z$ be a pl-flipping contraction over $V$ such that $\rho(X/Z)=1$. Assume that there exists a plt 6-complement $(X,S+B^c)$ of $(X,S+B)$ over $Z$. Then the flip exists.
\end{prop}
\begin{proof}
Write $K_{S^n}+B_{S^n}=(K_X+S+B)|_{S^n}$ and $K_{S^n}+B_{S^n}^c=(K_X+S+B^c)|_{S^n}$ for the normalisation $S^n$ of $S$. The pair $(S^n,B_{S^n}^c)$ is a klt 6-complement, so for any effective divisor $D$, $(S^n,B_{S^n}+\epsilon D)$ is globally $+$-regular for $0\le\epsilon\ll1$. In particular, the flip exists by \cite[Corollary 7.9, Theorem 8.25]{BMP+20}.
\end{proof}

The following proposition addresses Case (2).

\begin{prop}\label{prop: case 2}
 Let $(X,\Delta)$ be a 3-dimensional $\Qq$-factorial qdlt pair over $V$, $f:X\to Z$ be a $(K_X+\Delta)$-flipping contraction over $V$ such that $\rho(X/Z)=1$, and $\Sigma$ be a flipping curve. Assume that there exist distinct divisors $S,E\subseteq\lf\Delta\rf$ such that $S\cdot\Sigma<0$ and $E\cdot\Sigma<0$. Then the flip exists.
\end{prop}
\begin{proof}
This follows by exactly the same proof of \cite[Proposition 5.3]{HW19b}.
\end{proof}

Now, we deal with Case (3). Note that we will apply this proposition later in the case when $B$ does not have standard coefficients.

\begin{prop}\label{prop: case 3}
Let $(X,S+B)$ be a 3-dimensional $\Qq$-factorial plt pair over $V$ with $S$ irreducible. Let $f:X\to Z$ be a pl-flipping contraction over $V$ such that $\rho(X/Z)=1$ and $-S$ is relatively ample, and $\Sigma$ be a flipping curve. Assume that there exists a 6-complement $(X,S+E+B^c)$ of $(X,S+B)$ such that $E$ is irreducible, $E\cdot\Sigma\ge0$, and $E\cap\Sigma\neq\emptyset$. Then the flip exists. 
\end{prop}

\begin{proof}
Let $S^n$ be the normalisation of $S$. By perturbing the coefficients of $\lf B\rf$, we may assume that $(X,S+B)$ is plt. The pair $(S^n,B_{S^n})$ admits a 6-complement $(S^n,E|_{S^n}+B^c_{S^n})$, where $K_{S^n}+B_{S^n}=(K_X+S+B)|_{S^n}$ and $K_{S^n}+E|_{S^n}+B_{S^n}^c=(K_X+S+E+B^c)|_{S^n}$.

We claim that $E|_{S^n}$ is not exceptional over $Z$. Indeed, otherwise 
$$
0>(E|_{S^n})^2=E\cdot(E\cap S)=E\cdot\sum\lambda_i\Sigma_i\ge0
$$
for some flipping curves $\Sigma_i$ and some $\lambda_i>0$, which is a contradiction.

By Lemma \ref{lem: 6-complement with an integral boundary}, for any effective divisor $D$, the pair $(S^n,B_{S^n}+\epsilon D)$ is globally $+$-regular over $Z$ for any $0\le\epsilon\ll1$, and so the flip exists by \cite[Corollary 7.9, Theorem 8.25]{BMP+20}.
\end{proof}

\section{Divisorial extractions}

In this section we prove that we could extract a single divisorial place for 6-complements.

\begin{prop}\label{prop: existence of divisorial extractions}
Let $(X,\Delta)$ be a three-dimensional $\Qq$-factorial lc pair over $V$. Assume $X$ is klt and $6(K_X+\Delta)\sim0$. Let $E$ be a non-klt valuation of $(X,\Delta)$ over $X$. Then there exists a projective birational morphism $g:Y\to X$ such that $E$ is its exceptional locus.
\end{prop}

\begin{proof}
Let $\pi: Y\to X$ be a dlt modification of $(X,\Delta)$ such that $E$ is a divisor on $Y$ (see \cite[Corollary 4.9]{TY20}). Let $\Exc(\pi)=E+E_1+\cdots+E_m$. Write $K_Y+\Delta_Y=\pi^*(K_X+\Delta)$ 
and $K_Y+(1-\epsilon)\pi^{-1}_*\Delta+aE+a_1E_1+\cdots+a_mE_m=\pi^*(K_X+(1-\epsilon)\Delta),$ where $a,a_1,...,a_m<1$ as $X$ is klt, and set
$$
\Delta'=(1-\epsilon)\pi^{-1}_*\Delta+aE+E_1+\cdots+E_m.
$$
By taking $0<\epsilon\ll1$, we can assume that $a>0$. Note that 
\begin{align}\label{eqn: divisorial}
    K_Y+\Delta'\sim_{\Qq,X}(1-a_1)E_1+\cdots+(1-a_m)E_m,
\end{align}
so that the $(K_Y+\Delta')$-MMP over $X$ will not contract $E$ and the contracted loci are always contained in the support of the strict transform of $(1-a_1)E_1+\cdots+(1-a_m)E_m$. The negativity lemma implies that the output of a $(K_Y+\Delta')$-MMP over $X$ is the sought-for extraction of $E$. Hence, it is enough to show that we can run such an MMP.\par
By induction, we can assume that we have constructed the $n$-th step of this MMP $h:Y\dasharrow Y_n$ and we need to show that we can construct the $(n+1)$-st step. Let $\pi_n: Y_n\to X$ be the induced morphism, $\Delta'_n:=h_*\Delta'$, and $\Delta_n=h_*\Delta_Y$. By abuse of notation, we denote the strict transforms of $E,E_1,...,E_m$ by the same symbols.\par
The cone theorem is valid by \cite[Theorem 9.8]{BMP+20} (and also by \cite[Proposition 4.2]{TY20}). Let $R$ be a $(K_{Y_n}+\Delta'_n)$-negative extremal ray. By \eqref{eqn: divisorial}, we have $R\cdot E_i<0$ for some $i\ge1$. Then the contraction $f:Y_n\to Y'_n$ of $R$ exists by \cite[Theorem 9.10]{BMP+20} (and also by \cite[Propostion 4.1]{TY20}).\par
If $f$ is divisorial, then we set $Y_{n+1}:=Y'_n$. If $f$ is a flipping contraction, then the proof of \cite[Lemma 3.1]{HW19a} applied to $(Y_n,\Delta_n)$ over $X$ implies the existence of a divisor $E'\subseteq\Exc(\pi_n)$ such that $R\cdot E'>0$. Since $(Y_n,\Delta'_n)$ is dlt, $(Y_n,\Delta_n)$ is lc, $6(K_{Y_n}+\Delta_n)\sim_{\pi_n}0$, and $E'\le\Delta_n$, we can apply Proposition \ref{prop: case 3} to infer the existence of the flip of $f$.\par
The termination of this MMP follows by the usual special termination argument (see \cite[Proposition 4.5]{TY20} and \cite[Theorem 9.7]{BMP+20}). 
\end{proof}

\begin{cor}\label{cor: qdlt modification with unique extraction E}
Let $(X,S+B)$ be a three-dimensional $\Qq$-factorial plt pair defined over $V$. Assume that $X$ is klt, $S$ is a prime divisor and $(X,S+B)$ admits a 6-complement $(X,S+B^c)$ such that $(S^n,B^c_{S^n})$ has a unique non-klt place, where $K_{S^n}+B^c_{S^n}=(K_X+S+B^c)|_{S^n}$ and $S^n$ is the normalisation of $S$.\par
Then $(X,S+B^c)$ is qdlt in a neighborhood of $S$, or $\lf B^c\rf$ is disjoint from $S$ and there exists a projective birational map $\pi: Y\to X$ such that $(Y,S_Y+B^c_Y)$ is qdlt over a neighborhood of $S$, the exceptional divisor $E$ is irreducible and $E\subseteq\lf B^c_Y\rf$, where $K_Y+S_Y+B^c_Y=\pi^*(K_X+S+B^c)$.
\end{cor}

\begin{proof}
We work in a sufficiently small open neighbourhood of $S$. First, suppose that $\lf B^c\rf$ is non-empty and intersects $S$. Under this assumption the unique log canonical centre of $(S^n,B^c_{S^n})$ must be an irreducible curve given as $\lf B^c\rf|_{S^n}$. In particular, $\lf B^c\rf$ must be irreducible (cf. Remark \ref{rem: generic point of k-stratum has codim k}), the pair $(S^n,B^c_{S^n})$ is plt, and $(X,S+B^c)$ is qdlt by Lemma \ref{lem: inversion of adjunction for qdlt pair}.\par
Thus, we can assume that $\lf B^c\rf=0$, and so the dlt modification $\pi:Y\to X$ is nontrivial. Set $K_Y+\Delta^c_Y=\pi^*(K_X+S+B^c)$ and pick an irreducible exceptional divisor $E_1$ which is not an articulation point of  $D(\Delta_Y^{c,=1})$ (for example pick any divisor with the farthest distance edgewise in $D(\Delta_Y^{c,=1})$ from the node corresponding to $S$). Let $g:X_1\to X$ be the extraction of $E_1$ (see Proposition \ref{prop: existence of divisorial extractions}) and write 
$$
K_{X_1}+S_1+E_1+B^c_1=g^*(K_X+S+B^c)
$$
where $S_1,B^c_1$ are the strict transforms of $S,B^c$, respectively. Note that $S_1$ intersects $E_1$.\par
We claim that $(X_1,S_1,E_1+B^c_1)$ is qdlt in a neighbourhood of $S_1$. To this end we note that 
$$
K_{S^n_1}+B^c_{S^n_1}:=(K_{X_1}+S_1+E_1+B^c_1)|_{S^n_1}=g^*(K_{S^n}+B^c_{S^n}),
$$
where $S^n_1$ is the normalisation of $S_1$. Since $(S^n,B^c_{S^n})$ admits a unique non-klt place, we get that $(S^n_1,B^c_{S^n_1})$ is plt. In particular, Lemma \ref{lem: inversion of adjunction for qdlt pair} implies our claim.\par
Therefore, it is enough to show that $(X_1,S_1+E_1+B^c_1)$ does not admit a log canonical centre which is disjoint from $S_1$ and intersects $E_1$. By contradiction, assume that it does admit such a log canonical centre. Let $h:W\to X_1$ be a projective birational morphism which factors through $Y$
\begin{center}
\begin{tikzcd}
g\circ h: W \arrow[r, "h_Y"]
& Y \arrow[r, "\pi"] & X,
\end{tikzcd}
\end{center}
and such that $g\circ h$ is a log resolution of $(X,S+B)$. Write $K_W+\Delta_W^c=h^*(K_{X_1}+S_1+E_1+B^c_1)$. Since $S_1\cap E_1$ is disjoint from the other log canonical centres, the strict transform $E_{W,1}$ of $E_1$ is an articulation point of $D(\Delta^{c,=1}_W)$. Since $K_W+\Delta_W^c=h_Y^*(K_Y+\Delta_Y^c)$, Lemma \ref{lem: articulation point is stable between birational morphism} implies that $E_1$ is an articulation point of $D(\Delta_Y^{c,=1})$ which is a contradiction. In particular, $S_1,E_1$, and the irreducible curve $S_1\cap E_1$ are the only log canonical centres of $(X_1,S_1+E_1+B^c_1)$.
\end{proof}

\section{Existence of flips}
In this section we prove the main theorem. We start by showing the following result.

\begin{thm}\label{thm: existence of flip for klt pair with standard coef}
Let $(X,\Delta)$ be a three-dimensional $\Qq$-factorial klt pair with standard coefficients over $V$. Assume that $V$ is as in Setting \ref{setting of V} and additionally is a local ring with infinite residue field. If $f:X\to Z$ is a flipping contraction over $V$, then the flip $f^+:X^+\to Z$ exists.
\end{thm}

\begin{proof}
We will assume throughout that $Z$ is a sufficiently small affine neighborhood of $Q:=f(\Exc(f))$. We say that a $\Qq$-Cartier divisor $D$ is ample if it is relatively ample over $Z$.\par

By Shokurov’s reduction to pl-flips, it suffices to show the existence of pl-flips. Let $(X,S+B)$ be a plt pair with standard coefficients and $f:X\to Z$ a pl-flipping contraction. In particular $-S$ and $-(K_X+S+B)$ are $f$-ample, and so $\Exc(f)\subseteq S$. By \cite[Corollary 7.9, Theorem 8.25]{BMP+20}, the flip exists unless there exists an effective divisor $D$ such that $(S^n,B_{S^n}+\epsilon D)$ is not globally $+$-regular over $T=f(S)$ for any $\epsilon>0$, where $S^n$ is the normalisation of $S$ and $K_{S^n}+B_{S^n}=(K_X+S+B)|_{S^n}$. Thus we can assume that this is the case.\par
Proposition \ref{prop: m-complement for (X,S+B)} shows the existence of an $m$-complement $(X,S+B^c)$ of $(X,S+B)$ and Remark \ref{rem: m=6 when char p=5} implies that $m=6$. Let $(S^n,B^c_{S^n})$ be the induced 6-complement of $(S^n,B_{B^n})$. By Propositon \ref{prop: unique non-klt place for 6-complement}, the pair $(S^n,B^c_{S^n})$ has a unique place $C$ of log discrepancy zero which is exceptional over $T$.\par
If $(X,S+B^c)$ is qdlt, then the flip exists by Proposition \ref{Prop: Existence of flip for qdlt pair}. Thus, by Corollary \ref{cor: qdlt modification with unique extraction E}, we may assume that $\lf B^c\rf=0$ and there exists a qdlt modification $g:X_1\to X$ of $(X,S+B^c)$ with an irreducible exceptional divisor $E_1$. Let $f_1: X_1\to Z$ be the induced map to $Z$, and write $K_{X_1}+S_1+B_1+aE_1=g^*(K_X+S+B)$, and $K_{X_1}+S_1+B_1^c+E_1=g^*(K_X+S+B^c)$. In particular, $S_1\cap E_1$ is the unique log canonical place of $(S^n,B^c_{S^n})$, and so there are two possibilities: either $g(E_1)\subseteq S$ and $f_1(E_1)=Q$, or $g(E_1)\nsubseteq S$ is a curve intersecting $S$.\par
We would like to run a $(K_{X_1}+S_1+B_1+aE_1)$-MMP. It could possibly happen that $a<0$ so we take $0<\lambda\ll1$ and set 
$$
\Delta_1:=\lambda(S_1+B_1+aE_1)+(1-\lambda)(S_1+B^c_1+E_1),
$$
so that $K_{X_1}+\Delta_1\sim_{\Qq,Z}\lambda(K_{X_1}+S_1+B_1+aE_1)$, and $(X_1,\Delta_1)$ is plt.\par
Since $\rho(X/Z)=1$ and both $-(K_X+S+B)$ and $-S$ are ample over $Z$, it follows that $K_X+S+B\sim_{\Qq,Z}\mu S$ for some $\mu>0$ and so 
\begin{align}\label{eqn: existence of flips}
    K_{X_1}+\Delta_1\sim_{\Qq,Z}\lambda(K_{X_1}+S_1+B_1+aE_1)\sim_{\Qq,Z}\lambda\mu S_1+\lambda'E_1,
\end{align}
where $\lambda'\ge0$. Note $\lambda'>0$ if $g(E_1)\subseteq S$ and $\lambda'=0$ if $g(E_1)\nsubseteq S$.

\begin{claim}\label{claim: semiample over Z_q}
$S_1|_{X_{1,\Qq}}$ is semiample over $Z_{\Qq}$.
\end{claim}
\begin{proof}
Since  $g^*(S)=S_1+a_1E_1$ for some $a_1\ge0$ and $-E_1$ is $f$-ample over $X$, we see that $S_1$ is semi-ample over $X$. Notice that $X_{\Qq}=Z_{\Qq}$, thus the statement follows.
\end{proof}

\begin{claim}
There exists a sequence of $(K_{X_1}+\Delta_1)$-flips $X_1\dasharrow...
\dasharrow X_n$ over $Z$ such that either $X_n$ admits a $(K_{X_n}+\Delta_n)$-negative contraction of $E_n$ of relative Picard rank one, or $K_{X_n}+\Delta_n$ is semiample with the associated fibration contracting $E_n$. Here $\Delta_n$ and $E_n$ are strict transforms of $\Delta_1$ and $E_1$ respectively.
\end{claim}
In the course of the proof we will show that the qdlt-ness of $(X_1,S_1+E_1+B^c_1)$ is preserved (see Lemma \ref{lem: S_1 and S_2 has no intersection after flop if not qdlt}) except possibly at the very last step before the contraction takes place. Therefore, all the flips in this MMP exist by Proposition \ref{Prop: Existence of flip for qdlt pair}.
\begin{proof}
Let $f_i: X_i\to Z$ be the induced map to $Z$. We can assume that all the flipped curves are contracted to $Q\in Z$ under $f_i$, and so $X_1\dasharrow X_n$ is an isomorphism over $Z\backslash\{Q\}$. Let $(X_i,\Delta_i)$ and $(X_i,S_i+E_i+B_i^c)$ be the appropriate strict transforms. The latter pair is a 6-complement of $(X_i,S_i+E_i+B_i)$, where the strict transforms $B_i$ of $B_1$ have standard coefficients. Note that $E_1$ is not contracted as $X_1\dasharrow...\dasharrow X_n$ is a sequence of flips, thus inducing an isomorphism on the generic point of $E_1$.\par
Suppose that $K_{X_n}+\Delta_n$ is nef. There are two cases: either $g(E_1)\subseteq S$ and $f_1(E_1)=Q$, or $g(E_1)\nsubseteq S$. We claim that the former cannot happen. Indeed, assume that $f_1(E_1)=Q$ and let $\pi_1:W\to X_1$ and $\pi_n:W\to X_n$ be a common resolution of $X_1$ and $X_n$ such that $\pi_1$ and $\pi_n$ are isomorphisms over $Z\backslash Q$. Since $K_{X_n}+\Delta_n$ is nef and $K_{X_1}+\Delta_1$ is anti-nef (but not numerically trivial) over $Z$,
$$
\pi_n^*(K_{X_n}+\Delta_n)-\pi_1^*(K_{X_1}+\Delta_1)
$$
is exceptional, nef, and anti-effective over $Z$ by the negativity lemma. Moreover, its support must be equal to the whole exceptional locus over $Z$ as it is non-empty and contracted to $Q$ under the map to $Z$ (cf. \cite[Lemma 3.39(2)]{KM98}). This is impossible, because $E_1$ is not contained in its support while $f_1(E_1)=Q$.\par
Now, assuming that $g(E_1)\nsubseteq S$ is a curve intersecting $S$, we will show that $K_{X_n}+\Delta_n\sim_{\Qq,Z}\lambda\mu S_n$ is semiample. Let $G:=f_n^{-1}(P)$ for a (non-necessarily closed) point $P\in Z$. By \cite[Theorem 1.2]{Wit21} it is enough to show that $S_n|_{G}$ is semiample and $S_n|_{X_{n,\Qq}}$ is semiample over $Z_{\Qq}$. The latter follows from Claim \ref{claim: semiample over Z_q}. For the former, since $X_1\dasharrow X_n$ is an isomorphism over $Z\backslash \{Q\}$, $S_1=g^*S$, and $S$ is semiample over $Z\backslash\{Q\}$, we get that $S_n|_{G}$ is semiample when $P\neq Q$. Thus we may assume that $P=Q$. Since $G$ is a projective variety over a positive characteristic field, by \cite{Kee99} it is enough to verify that $S_n|_{\mathbb{E}(S_n|_G)}$ is semiample. Since $G$ is one-dimensional, every connected component of $\mathbb{E}(S_n|G)\subseteq G$ is either entirely contained in $S_n$ or is disjoint from it. In particular, it is enough to show that $(K_{X_n}+\Delta_n)|_{S_n}$ is semiample. Recall that $S_n\subseteq\lf\Delta_n\rf$, and so $K_{S^n_n}+\Delta_{S^n_n}=(K_{X_n}+\Delta_n)|_{S^n_n}$ is semiample by \cite[Theorem 4.2]{Tan18}, where $S^n_n$ is the normalisation of $S_n$. Since $S^n_n\to S_n$ is a universal homeomorphism (see \cite[Lemma 2.28]{BMP+20}), then by Lemma \ref{lem: semiample can be checked up to uni. homeo. in mixed char} $(K_{X_n}+\Delta_n)|_{S_n}$ is semiample and so is $K_{X_n}+\Delta_n$.\par
Since $(K_{X_n}+\Delta_n)|_{E_n}$ is relatively numerically trivial over $Z\backslash\{Q\}$ (as so is $(K_{X_1}+\Delta_1)|_{E_1}$), we get that the associated semiample fibration contracts $E_n$.
\vspace{1em}

From now on, $K_{X_n}+\Delta_n$ is not nef. In order to run the MMP, we assume that $(X_n, S_n+E_n+B^c_n)$ is qdlt by induction. The cone theorem is valid by \cite[Theorem 9.8]{BMP+20} (also by \cite[Proposition 4.2]{TY20}). Pick $\Sigma_n$ a $(K_{X_n}+\Delta_n)$-negative extremal curve. By \eqref{eqn: existence of flips}, we have $K_{X_n}+\Delta_n\sim_{\Qq,Z}\lambda\mu S_n+\lambda'E_n$, and thus either $\Sigma_n\cdot S_n<0$ or $\Sigma_n\cdot E_n<0$. The contraction of $\Sigma_n$ exists by \cite[Theorem 9.10]{BMP+20} (also by \cite[Proposition 4.1]{TY20}) applied to $(X_n,\Delta_n)$ in the former case and to $(X_n+(1-\epsilon)S_n+E_n+B_n)$ in the latter for $0<\epsilon\ll1$.\par
If the corresponding contraction is divisorial, then we are done as it must contract $E_n$. Hence, we can assume that $\Sigma_n$ is a flipping curve. If $E_n\cdot\Sigma_n\le0$, then $-(K_{X_n}+S_n+B_n+E_n)$ has standard coefficients, is qdlt and ample over the contraction of $\Sigma_n$, so the flip exists by Proposition \ref{Prop: Existence of flip for qdlt pair} as $(X_n,S_n+E_n+B^c_n)$ is a 6-complement. If $E_n\cdot\Sigma_n>0$, then the flip exists by Proposition \ref{prop: case 3} applied to $(X_n,\Delta_n)$.
\vspace{1em}

To conclude the proof we shall show that $(X_{n+1},S_{n+1}+E_{n+1}+B^c_{n+1})$ is qdlt unless $X_{n+1}$ admits a contraction of $E_{n+1}$. By Lemma \ref{lem: S_1 and S_2 has no intersection after flop if not qdlt}, we can suppose that $S_{n+1}\cap E_{n+1}=\emptyset$ and aim for showing that the sought-for contraction exists.\par
Let $\Sigma'$ be a curve which is exceptional over $Q\in Z$, contained neither in $S_{n+1}$ nor $E_{n+1}$, but instersecting $S_{n+1}$ (it exists by connectedness of the exceptional locus over $Q\in Z$, and the fact that both $S_{n+1}$ and $E_{n+1}$ intersect this exceptional locus), and let $C\subseteq E_{n+1}$ be any exceptional curve such that $C\cdot E_{n+1}<0$ (it exists by the negativity lemma as $E_{n+1}$ is exceptional over $Z$). We claim that $C'\cdot S_{n+1}>0$ for every exceptional $C'\nsubseteq E_{n=1}$. To this end, assume by contradiction that there exists $C'\nsubseteq E_{n+1}$ satisfying $C'\cdot S_{n+1}\le0$. Since $\rho(X_{n+1}/Z)=2$, we get that 
$$
C'\equiv aC+b\Sigma'
$$
for $a,b\in \mathbb{R}$. Given $C\cdot S_{n+1}=0$ and $\Sigma'\cdot S_{n+1}\ge0$, we have $b\le 0$. As $C'\cdot E_{n+1}\ge0,~C\cdot E_{n+1}<0$, and $\Sigma'\cdot E_{n+1}\ge0$, we have $a\le0$. Therefore, for an ample divisor $A$ we have 
$$
0<C'\cdot A=(aC+b\Sigma')\cdot A\le0
$$
which is a contradiction.\par
Since $S_{n+1}\cap E_{n+1}$ is empty, $S_{n+1}$ is thus nef and $\mathbb{E}(S_{n+1})\subseteq E_{n+1}$. Hence $S_{n+1}$ is semiample by Claim \ref{claim: semiample over Z_q} and \cite[Theorem 6.1]{Wit20} and induces a contraction of $E_{n+1}$. It does not contract $\Sigma'$, and so is of relative Picard rank one. Moreover, by \eqref{eqn: existence of flips} we have either $\lambda'=0$ and $K_{X_{n+1}}+\Delta_{n+1}\sim_{\Qq,Z}\lambda\mu S_{n+1}$ is semiample with the associated fibration contracting $E_{n+1}$, or $\lambda'>0$, $(K_{X_{n+1}}+\Delta_{n+1})\cdot C<0$, and so the above contraction is a $(K_{X_{n+1}}+\Delta_{n+1})$-negative Mori contraction of relative Picard rank one.
\end{proof}

Let $\phi: X_n\to X^+$ be the contraction of $E_{n}$ as in the previous claim, let $\Delta^+:=\phi_*\Delta_n$, let $S^+:=\phi_* S_n$, and let $B^+:=\phi_*B_n$. Then the induced map $\pi^+:X^+\to Z$ is a small contraction with $\rho(X^+/Z)\le 1$. Recall that 
$$
K_{X_n}+\Delta_n\sim_{\Qq,Z}\lambda(K_{X_n}+S_n+aE_n+B_n).
$$
Since $\phi$ is either $(K_{X_n}+S_n+aE_n+B_n)$-negative of Picard rank one or $(K_{X_n}+S_n+aE_n+B_n)$-trivial , the discrepancies of $(X^+,S^++B^+)$ are not smaller than those of $(X_n,S_n+aE_n+B_n)$. Moreover, since $(K_{X_1}+S_1+aE_1+B_1)$ is anti-nef over $Z$ and not numerically trivial, at least one step of the $(K_{X_1}+S_1+aE_1+B_1)$-MMP has been performed in $X_1\dasharrow X^+$. In particular, there exists a divisorial valuation for which the discrepancy of $(X^+,S^++B^+)$ is higher than the discrepancy of $(X,S+B)$.\par
Therefore, $K_{X^+}+\Delta^+$ cannot be relatively anti-ample, because then\\ $(X^+,S^++B^+)$ would be isomorphic to $(X,S+B)$, which is impossible as the MMP has increased the discrepancies. If $K_{X^+}+\Delta^+$ is relatively numerically trivial, then we claim that $K_{X^+}+\Delta^+\sim_{\Qq,Z}0$. Indeed
$$
K_{X^+}+\Delta^+\sim_{\Qq,Z}\lambda\mu S^+,
$$
for $\lambda,\mu>0$, and since $S^+$ intersects the exceptional locus, we must in fact have that $\Supp\Exc(\pi^+)$. By \cite[Theorem 1.2]{Wit20}, it is thus enough to show $K_{S^{+,n}}+\Delta_{S^{+,n}}$ is semiample, where $S^{+,n}\to S^+$ is the normalisation of $S^+$, which in turn follows from \cite[Theorem 4.2]{Tan18}. Here we used the fact that $S^{+,n}\to S^+$ is a universal homeomorphism (see \cite[Lemma 2.28]{BMP+20}). As a consequence, $S^+$ descends to $Z$. This is impossible as its image in $Z$ is not $\Qq$-Cartier.\par
Thus $K_{X^+}+\Delta^+$ is relatively ample, and so $X^+\to Z$ is the flip of $X\to Z$ by \cite[Corollary 6.4]{KM98}.  
\end{proof}

Given Theorem \ref{thm: existence of flip for klt pair with standard coef}, we can follow the same strategy as in \cite[Theorem 6.3]{Bir16} to move the ``standard coefficients" condition (cf. \cite[Theorem 9.12]{BMP+20}).

\begin{prop}\label{prop: existence of flips when the residue filed is infinite}
Theorem \ref{thm: existence of flip for klt pairs} holds when in addition $V$ is a local ring whose residue field is infinite. 
\end{prop}

\begin{proof}
First, we can assume that every component $S$ of $\Supp\Delta$ is relatively anti-ample. Further, let $\zeta(\Delta)$ be the number of components of $\Delta$ with coefficients not in the set $\Gamma:=\{1\}\cup\{1-\frac{1}{n}~|~n>0\}$. If $\zeta(\Delta)=0$ then the flip exists by Theorem \ref{thm: existence of flip for klt pair with standard coef}. By induction, we can assume that the flip exists for all flipping contractions of log pairs $(X',\Delta')$ such that $\zeta(\Delta')<\zeta(\Delta)$.\par
By replacing $\Delta$ with $\Delta-\frac{1}{l}\lf\Delta\rf$ for $l\gg0$, we can assume $(X,\Delta)$ is klt without changing $\zeta(\Delta)$. Write $\Delta=aS+B$, where $S\nsubseteq\Supp B$ and $a\notin\Gamma$. Let $\pi:W\to X$ be a log resolution of $(X,S+B)$ with exceptional divisor $E$ and set $B_W:=\pi_*^{-1}B+E$. Since $K_X+\Delta\equiv_Z\mu S$ for some $\mu>0$, we have
\begin{align*}
    K_W+S_W+B_W &=\pi^*(K_X+\Delta)+(1-a)S_W+F \\
    &\equiv_Z(1-a+\mu)S_W+F',
\end{align*}
where $S_W:=\pi_*^{-1}S$, and $F,F'$ are effective $\Qq$-divisors exceptional over $X$.\par
Run a $(K_W+S_W+B_W)$-MMP over $Z$. By induction all flips exist in this MMP as $\zeta(S_W+B_W)<\zeta(\Delta)$. Moreover, by the above equation, every extremal ray is negative on $(1-a+\mu)S_W+F'$ and hence on an irreducible component of $\lf S_W+B_W\rf$. In particular, all contractions exist by \cite[Theorem 9.10]{BMP+20} (also by \cite[Proposition 4.1]{TY20}). The cone theorem is valid by \cite[Theorem 9.8]{BMP+20} (also by \cite[Proposition 4.2]{TY20}) and this MMP will terminate by the special termination (cf. \cite[Proposition 4.5]{TY20} and \cite[Theorem 9.7]{BMP+20}). Let $h:W\dasharrow Y$ be an output of this MMP and let $S_Y,B_Y$ and $F_Y$ be the strict transforms of $S_W,B_W$ and $F$ respectively.\par
Now, run a $(K_Y+aS_Y+B_Y)$-MMP with scaling of $(1-a)S_Y$. In particular, if $R$ is a $(K_Y+aS_Y+B_Y)$-negative extremal ray, then $R\cdot S_Y>0$ and this MMP is also a $(K_Y+B_Y)$-MMP. As $\zeta(B_Y)<\zeta(\Delta)$, all the flips in this MMP exist by induction. By the same argument as in the above paragraph, the cone theorem is valid, all contractions exist and this MMP will terminate. Let $(X^+,aS^++B^+)$ be an output of this MMP. We claim that this is the flip of $(X,aS+B)$.\par
To this end, we notice that the negativity lemma applied to a common resolution $\pi_1:W'\to X$ and $\pi_2:W'\to X^+$ implies that 
$$
\pi_1^*(K_X+aS+B)-\pi_2^*(K_{X^+}+S^++B^+)\ge0
$$
Since $(X,aS+B)$ is klt, this shows that $\lf B^+\rf=0$ and all the divisor in $E$ were contracted. In particular, $X\dasharrow X^+$ is an isomorphism in codimension one. We claim $K_{X^+}+aS^++B^+$ is relatively ample over $Z$ and so $(X^+,aS^++B^+)$ is the flip of $X\to Z$.\par
To this end, we note that $\rho(X^+/Z)=1$ (cf. \cite[Lemma 1.6]{AHK07}). Indeed,
$$
\rho(W'/X^+)+\rho(X^+/Z)=\rho(W'/X)+\rho(X/Z)
$$
and $\rho(W'/X)=\rho(W'/X^+)$ is equal to the number of exceptional divisors. Thus $\rho(X^+/Z)=\rho(X/Z)=1$. In particular, to conclude the proof of the theorem it is enough to show that $K_{X^+}+aS^++B^+$ cannot be relatively numerically trivial over $Z$. Assume by contradiction that it is relatively numerically trivial. Then 
$$
\pi_1^*(K_X+aS+B)-\pi_2^*(K_{X^+}+S^++B^+)\ge0
$$
is exceptional and relatively numerically trivial over $X$. Thus it is empty by the negativity lemma. Then $\pi_1^*(K_X+aS+B)\equiv_Z0$, which contradicts the fact that $K_X+aS+B$ is anti-ample over $Z$. 
\end{proof}

Now Theorem \ref{thm: MMP with scaling}, Theorem \ref{thm: Base point free theorem} and Theorem \ref{thm: Cone theorem} hold if we additionally assume that $V$ is a local ring with infinite residue field, by exactly the same proof of \cite[Theorem 9.34 and 9.36]{BMP+20}, \cite[Theorem 9.26]{BMP+20} and \cite[Theorem 9.27]{BMP+20} respectively.

\begin{proof}[Proof of Theorem \ref{thm: existence of flip for klt pairs}]
We can work over a small neighborhood of $f(\Exc(f))$, and the existence of the flip is equivalent to the finite generation of the graded algebra $\bigoplus_{m\ge0} f_*\Oo_X(m(K_X+\Delta))$ over $\Oo_Z$. This property is stable under localization by Lemma \ref{lem: f.g. stable under localization}. Hence we can assume that $V=\spec R$, where $R$ is an excellent DVR.\par

Let $R'$ be the completion of strict Henselization of $R$. Consider the base change $f':X'\to Z'$ of $f:X\to Z$. Since the residue field of $R'$ is now infinite, and the Minimal Model Program holds in this case, we get that $\bigoplus_{m\ge0} f'_*\Oo_{X'}(m(K_{X'}+\Delta'))$ is finitely generated over $\Oo_{Z'}$, where $K_{X'}+\Delta'$ is the pullback of $K_X+\Delta$ on $X'$. Since $Z'\to Z$ is faithfully flat, then $\bigoplus_{m\ge0} f_*\Oo_X(m(K_X+\Delta))$ is also finitely generated over $\Oo_Z$.
\end{proof}

\begin{lem}\label{lem: f.g. stable under localization}
Let $s \in Z$ be a closed point, and $Z_s := \cO_{Z, s}$. Suppose $D \subseteq Z$ is a divisor such that $\bigoplus_{m \ge 0} \cO_{Z_s}(mD_s)$ is finitely generated $\cO_{Z_s}$-algebra, where $D_s$ is the pullback of $D$ to $Z_s$. Then $\bigoplus_{m \ge 0} \cO_Z(mD)$ is a finitely generated $\cO_Z$-algebra in a neighborhood of $Z_s$.
\end{lem}

\begin{proof}
By \cite[Lemma 6.2]{KM98}, there exists a small projective birational morphism $g_s: Y_s \to Z_s$ such that $Y_s$ is normal and $g_s^* D_s$ is $g_s$-ample. Taking the closure of $g_s$, we get a projective morphism $g: Y \to Z$. Then there is an open subset $U$ contains $Z_s$ such that $g_U$ is small, $Y_U$ is normal and $g_U^* D_{s, U}$ is $g_U$-ample, where $D_{s, U}$ is the restriction of the closure of $D_s$ to $U$. Possibly shrinking $U$ we may assume that $D_{s, U}$ is exactly $D_U$, the restriction of $D$ to $U$. Hence by \cite[Lemma 6.2]{KM98}, $\bigoplus_{m \ge 0} \cO_Z(mD)$ is a finitely generated $\cO_Z$-algebra over $U$.
\end{proof}

Now Theorem \ref{thm: MMP with scaling}, Theorem \ref{thm: Base point free theorem} and Theorem \ref{thm: Cone theorem} follow again from \cite[Theorem 9.34 and 9.36]{BMP+20}, \cite[Theorem 9.26]{BMP+20} and \cite[Theorem 9.27]{BMP+20} respectively.

Finally, Theorem \ref{thm: Termination of flips} follows by the same proof of \cite[Proposition 9.18]{BMP+20} when $K_X+\Delta$ is pseudo-effective, and \cite[Corollary 3.5]{Sti21} when $K_X+\Delta$ is not pseudo-effective.

\end{document}